\documentclass[11pt,twoside]{article}
\usepackage[margin=1in]{geometry}
\usepackage{graphicx}
\usepackage{titlesec}
\usepackage{amsmath}
\usepackage{amsfonts}   
\usepackage{amssymb}    
\usepackage{amsthm}
\usepackage{mathtools}
\usepackage{verbatim}   
\usepackage{color}
\usepackage{listings}

\newcommand{\vect}[1]{\mathbf{#1}}

\newcommand{\x}{\vect{x}}

\newtheorem{thm}{Theorem}
\newtheorem{prop}{Proposition}

\numberwithin{equation}{section}

\title{Note on  Bessel functions of type $A_{N-1}$.}
\author{B\'{e}chir Amri}
\date{\small
University of Tunis, Preparatory Institut of Engineer
Studies of Tunis, Department of Mathematics,
1089 Montfleury Tunis, Tunisia\\ bechir.amri@ipeit.rnu.tn }
\begin{document}
\maketitle
\begin{abstract}
Through the theory of Jack polynomials
we give an iterative method   for  integral formula of Bessel function  of type $A_{N-1}$ and  a partial product formula for it.
\footnote{\par\textbf{Key words and phrases:} Dunkl operators, Heckman-Opdam polynomials, Jack polynomials.
\par\textbf{2010 Mathematics Subject Classification:}  33C52,33C67, 05E05.
\par Author  partially supported by DGRST project 04/UR/15-02 and CMCU program 10G 1503.}
\end{abstract}

\section{Introduction and backgrounds}
Dunkl operators which were first introduced by C. F. Dunkl \cite{D1} in the  late 80ies are commuting differential-difference operators, associated  to a finite reflection groups on a Euclidean space. Their  eigenfunctions are called Dunkl kernels and  appear  as a  generalization  of the exponential functions. Although  attempts were made to study them and except the reflection group $\mathbb{Z}_2^N$ the explicit forms or behaviors of these kernels  are remain unknown.
In the present work we will be concerned with generalized Bessel functions $J_k$ defined through symmetrization of Dunkl kernels  in the case of the symmetric group $S_N$. We will obtain the following
\begin{eqnarray}\label{0}
J_k(\mu,\lambda)=\int_{\mathbb{R}^{N-1}}e^{\langle\mu,x\rangle}\delta_k(\lambda,x)dx.
\end{eqnarray}
 where the function $\delta_k$ can be explicitly computed using a recursive  formula on the dimension $N$.
The key ingredient  is the integral formula of A. Okounkov and G. Olshanski \cite{AG} for
Jack polynomials. As   the last are connected with Heckman-opdam-Jacobi polynomials \cite{Br} the  formula (\ref{0}) follow  by limit transition. We should note here that when  $N=3$, the formula ($\ref{0}$) is comparable to that obtained by C. F. Dunkl \cite{D2} for intertwining operator.
 \par Let us start with some well-known facts about
Heckman Opdam  Jacobi polynomials, Jack polynomials and Dunkl kernels associated with a root system $R$.
The  standard references are \cite{Br,DJ1, Hec, OH1,S, R2}.
Here $ \mathbb{R}^N$ is equipped  with the  usual  inner  product $\langle,\; .\;,\rangle$ and  the canonical orthonormal basis
 $(e_1,e_2,...,e_N)$. Further, we shall assume that $R$ is reduced and  crystallographic, that is a finite subset of $\mathbb{R}^N\backslash\{0\}$ which satisfies:
\\(i) $R$ spanned  $\mathbb{R}^N$.
\\(ii) $R$ is invariant under $r_\alpha$ the reflection in the hyperplane orthogonal to any $\alpha\in R$.\\
(iii) $\alpha.\mathbb{R}\cap R=\{\pm \alpha\}$ for all $\alpha\in R$\\
(iii) for all $\alpha,\;\beta\in R$; $\langle \alpha,\breve{\beta}\rangle\in \mathbb{Z}$, $\breve{\beta}=\frac{2\beta}{\|\beta\|^2}$\\
 We assume that the reader is familiar with the basics of root systems and their Weyl groups, see  for examples Humphreys \cite{Hu}.
\subsection{ Heckman Opdam Jacobi polynomials.}
 Let  $R$ be a reduced  root system with  $\{\alpha_1,...,\alpha_N\}$ be a basis of simple roots and  $R_+$ be the set of positive
roots determined by this basis. The fundamental weights    $\{\beta_1,...,\beta_N\}$ are given  by
   $\langle\;\beta_j,\;\check{\alpha}_i\;\rangle=\delta_{i,j}$, $\displaystyle{\check{\alpha_i}=\frac{2\alpha_i}{\|\alpha_i\|^2}}$.
Let $\displaystyle Q=\bigoplus_{i=1}^N\mathbb{Z}\alpha_i$,  $\displaystyle P=\bigoplus_{i=1}^N\mathbb{Z}\beta_i$, $\displaystyle Q^+=\bigoplus_{i=1}^N\mathbb{N}\alpha_i$ and $\displaystyle P^+=\bigoplus_{i=1}^N\mathbb{N}\beta_i$.
  We define a  partial ordering
   on $P$   by   $\lambda\preceq\mu$   if $\mu-\lambda\in Q^+$
   \par The group algebra  $\mathbb{C}[P]$ of the free Abelian group $P$ is the algebra  generated
by the formal exponentials $e^\lambda$, $\lambda\in P$ subject to the multiplication relation $e^\lambda e^\mu=e^{\lambda+\mu}$.
 The Weyl group W acts on $\mathbb{C}[P]$ by $we^{\lambda}=e^{w\lambda}$. The orbit-sums  $\displaystyle m_\lambda=\sum_{\mu\in W.\lambda}e^\mu$, $\lambda\in P^+$ form  a basis of $\mathbb{C}[P]^W$, the subalgebra of $W$-invariant elements of $\mathbb{C}[P]$.  Here  $W.\lambda$ denotes  the  W-orbit of $\lambda$.
\par Let $\mathbb{T}=\mathbb{R}^d/2\pi\check{Q}$ where $\displaystyle\check{Q}=\bigoplus_{i=1}^d\mathbb{Z}\check{\alpha_i}$.  The algebra $\mathbb{C}[P]$ can be realized explicitly as the algebra of
polynomials on the torus $\mathbb{T}$ through the identification
 $e^\lambda(\dot{x})=e^{i\langle\lambda,x\rangle}$ where $\dot{x}\in \mathbb{T}$ is the image of $ x\in\mathbb{R}^d$.  Let $k: R\rightarrow[0,+\infty[$
   be a $W$-invariant  function, called multiplicity function.
  We equip $C[P]^W$ with the  inner  product
      $$(f,g)_k=\int_\mathbb{T}f(x)\overline{g(x)}\delta_k(x)dx$$
  where
$$\delta_k=\prod_{\alpha\in R^+}\left|e^\frac{\alpha}{2}-e^{-\frac{\alpha}{2}}\right|^{2k_\alpha}$$
and $dx$ is  the Haar measure on $\mathbb{T}$.
     \par The Heckman Opdam Jacobi polynomials    are introduced by Heckman and Opdam
 \cite{Hec} as
      the unique family of elements $P_\lambda\in\mathbb{C}[P]^W$, $\lambda\in P^+$
    satisfying the following conditions:
    \begin{itemize}
      \item [(i)]  $P_\lambda=m_\lambda+\sum_{\mu\prec\lambda}a_{\lambda\mu}m_\mu$
      \item [(ii)]  $\langle P_\lambda,m_\mu\rangle=0$     if $\mu\in P_+$, $\lambda\prec\mu$.
    \end{itemize}
  ( Note that in \cite{Hec}, these polynomials  are indexed by $-P_+$ instead of $P_+$ ). They form an orthogonal basis of $\mathbb{C}[P]^W$ and satisfy the second differential equation
     $$\Big(\Delta+\sum_{\alpha\in R_+}k_\alpha\coth(\frac{1}{2}\langle x,\alpha\rangle)\partial_\alpha\Big)P_\lambda(x)=
     \langle\lambda,\lambda+\sum_{\alpha\in R_+}k_\alpha \alpha\rangle P_\lambda(x).$$
where $\Delta$ is the Laplace operator on $\mathbb{R}^N$.
\par  The Cherednik operator $T_\xi$, $\xi\in\mathbb{R}^N$, associated with the root system $R$ and the multiplicity $k$ is defined by
$$T_\xi^k=\partial_\xi +\sum_{\alpha\in R_+}k_\alpha\langle\alpha,\;\xi\rangle\;\frac{1-r_\alpha}{1-e^{^{\alpha}}}-\langle\rho_k,\;\xi\rangle, $$
where $\displaystyle{\rho_k=\frac{1}{2}\sum_{\alpha\in R_+}k_\alpha \alpha}$.
 The hypergeometric function $F_k$ is defined  as
 the unique holomorphic  W-invariant function  on  $\mathbb{C}^N\times (\mathbb{R}^N+iU)$ ( U is a W-invariant neighborhood  of $0$ ) which satisfies
 the system of differential equations:
 $$p(T_{e_1},...T_{e_N})F_k(\lambda,.)=p(\lambda)F_k(\lambda,.); \qquad F(\lambda,0)=1$$
for all $\lambda\in \mathbb{C}^N$ and all $W$-invariant polynomial $p$ on $\mathbb{R}^N$. The Heckman opdam Jacobi polynomials are related to the hypergeometric function  $F_k$ by  ( see \cite{HS} )
          \begin{eqnarray}\label{c}
         F_k(\lambda+\rho_k,x)= c(\lambda+\rho_k)P_\lambda(x);\qquad\lambda\in P^+,\; x\in\mathbb{R}^N,
           \end{eqnarray}
          where the function $c$ is given  on $\mathbb{R}^N$ by
\begin{eqnarray}\label{har}
c(\lambda)= \prod_{\alpha\in R^+}\frac{\Gamma(\langle\lambda,\check{\alpha}\rangle)\Gamma(\langle\rho,\check{\alpha}\rangle+k_\alpha)}{\Gamma(\langle\lambda,\check{\alpha}\rangle+k_\alpha)
          \Gamma(\langle\rho,\check{\alpha}\rangle)}.
\end{eqnarray}
\subsection{ Jack polynomials}
Let $k>0$, the symmetric group $S_N$  acts on the ring of polynomials  $\mathbb{Q}(k)[x_1,...,x_N]$ by
$$\tau p(x_1,...,x_N)=p(x_{\tau(1)},...,x_{\tau(N)})$$
  Let $\Lambda_N$ the subspace of symmetric polynomials,
$$\Lambda_N=\{p\in \mathbb{C}[x_1,...,x_N],\; \tau p=p,\; \forall \tau \in S_N\}.$$
We call partition all $\lambda=(\lambda_1,...\lambda_N)\in \mathbb{N}^N$ such that
$\lambda_1\geq...\geq\lambda_N$. The weight of a partition $\lambda$ is the sum
$|\lambda|=\lambda_1+...+\lambda_N$ and its length $\ell(\lambda)=\max\left\{j;\;\lambda_j\neq 0\right\}.$
The set of all partitions  are partially ordered by the dominance order:
$$\lambda\leq \mu \Leftrightarrow|\lambda|=|\mu|\quad\text{ and}\quad
\lambda_1+\lambda_2+...+\lambda_i\leq \mu_1+\mu_2+...+\mu_i$$
 for all $i=1,2,...,N$. The simplest basis of $\Lambda_N$ is given by the monomial symmetric polynomials,
$$m_{\lambda}(x)=\sum_{\mu\in S_N \lambda}x_1^{\mu_1}...x_N^{\mu_N}.$$
 We define an inner product on $\Lambda_N$ by
$$\langle f,g\rangle_k=\int_Tf(z)\overline{g(z)}\prod_{i<j}|z_i-z_j|^{2k}dz$$
where  $T=\{(z_1,...,z_N)\in \mathbb{C}^N; |z_j|=1,\;\forall \;1\leq j \leq N\}$ is the $N$-dimensional torus and $dz$ is the haar measure on $T$.
Jack symmetric polynomials $j_\lambda$ indexed by a partitions  $\lambda$  can be defined as the unique polynomials such that
\begin{itemize}
  \item [(i)] $j_\lambda = m_\lambda +\sum_{\mu\prec\lambda}m_\mu$,
  \item [(ii)] $\langle j_\lambda,m_\mu\rangle_k=0$ if $\lambda\leq\mu$.
\end{itemize}
By a result of I. G. Macdonald (\cite{M}, p: 383 ) they form a family of orthogonal polynomials. Jack polynomials can be defined as eigenfunctions of certain Laplac-Beltrami type operator ( coming in the theory of Calogero integrable systems and in random matrix theory ),
$$
L_k=\sum_{i=1}^dx_i^2\frac{\partial^2}{\partial x_i^2}+2k\sum_{i\neq j}\frac{x_i^2}{x_i-x_j}\frac{\partial}{\partial x_i}.$$
Jack polynomials $j_\lambda$ are homogeneous of degree $|\lambda|$ and satisfy  the compatibility relation
\begin{eqnarray}\label{cr}
j_{(\lambda_1,...,\lambda_{N-1},0)}(x_1,...,x_{N-1},0)=j_{(\lambda_1,...,\lambda_{N-1})}(x_1,...,x_{N-1}).
\end{eqnarray}
\par The relationship between Heckman Opdam Jacobi polynomials and Jack polynomials can be illustrated as  follows ( see \cite{Br} ):  Let  $\mathbb{V}$  be  the  hyperplane    orthogonal  to  the  vector $e=e_1+...+e_N$. In  $V$  we consider  the
root  system  of  type  $A_{N-1}$,
$$R_A=\{\pm(e_i-e_j),\; 1\leq i<j\leq N\}.$$
The  fundamental  weights  are given by $\pi_N(\omega_i)$, $\omega_i=e_1+...+e_i$, where   $\pi_N$  denote the orthogonal projection  along  $e$ onto  V,
$$\pi_N(x)=x-\frac{1}{N}\left(\sum_{i=1}^Nx_i\right)e=\left(x_1-\frac{1}{N}
\left(\sum_{i=1}^Nx_i\right),...,x_N-\frac{1}{N}\left(\sum_{i=1}^Nx_i\right)\right)$$
and then  $P_A^+=\{\pi_N(\lambda), \lambda \;\text{partition} \}$. The result is that:
\begin{eqnarray}\label{je}
j_\lambda(e^{x})=P_{\pi_N(\lambda)}(x),
\end{eqnarray}
For all  partition $\lambda$ and all $x\in\mathbb{V}$ with $e^x=(e^{x_1},...,e^{x_N})$.
\subsection{Dunkl kernels and  Dunkl-Bessel functions }
The Dunkl operator $D_\xi$, $\xi\in \mathbb{R}^N$ associated with a root system $R$ and a multiplicity function $k$ is defined by
$$D_\xi=\partial_\xi+\sum_{\alpha\in R^+}k(\alpha)\langle\alpha,\xi\rangle\frac{1-r_\alpha}{\langle \alpha,.\rangle}.$$
The Dunkl intertwining operator   $V_k$ is the unique isomorphism on the  polynomials space $\mathbb{C}[\mathbb{R}^N]$ such that
$$V_k(1)=1,\quad V_k(\mathcal{P}_n)=\mathcal{P}_n\quad\text{and}\quad D_\xi V_k=V_k\partial_\xi$$
where $\mathcal{P}_n$ is the subspace of homogeneous polynomials of degree $n\in \mathbb{N}$. For $r>0$ , $V_k$
 extends to a continuous linear operator on the Banach space
$$A_r=\{f=\sum_{n=0}^\infty f_n,\; f_n\in \mathcal{P}_n,\; \|f\|_{A_r}= \sum_{n=0}^\infty \sup_{|x|\leq r}|f_n(x)|<\infty\}$$
by  $$V_k(f)=\sum_{n=0}^\infty V_k( f_n).$$
A remarkable result due to M. R\"{o}sler \cite{R3} says that
for each $x \in \mathbb{R}^N$,
$$V_k(f)(x)=\int_{\mathbb{R}^d}f(\xi)d\mu_x(\xi)$$
where $\mu_x$ is a probability measure supported  in $co(x)$ the convex hull of the orbit W.x.
\par The Dunkl kernel $E_k$ is given by
\begin{eqnarray*}
E_k(x,y)=V_k(e^{\langle \;.\;,\;y\;\rangle})(x)=\int_{\mathbb{R}^d}e^{\langle\xi,y\rangle}d\mu_x(\xi),\quad x\in\mathbb{R}^N,\; y\in \mathbb{C}^N
\end{eqnarray*}
and having the following properties:
\begin{itemize}
  \item [(i)]For each $y\in \mathbb{C}^N$ the function $E_k(.,y)$ is the  unique solution of eigenvalue problem:
$$D_\xi f(x)=\langle\xi,y\rangle f(x)\;\forall\;\quad\xi\in \mathbb{R}^N\;\text{and}\;f(0)=1.$$
\item [(ii)] $E_k$ extends to a holomorphic  function on $\mathbb{C}^N\times\mathbb{C}^N$ and for all $(x,y)\in\mathbb{C}^N\times\mathbb{C}^N$,
$w\in W$ and $t\in\mathbb{C}$: $$E_k(x,y)=E_k(y,x),\quad E_k(wx,wy)=E_k(x,y)\quad\text{and}\quad E_k(x,ty)=E_k(tx,y)$$
\end{itemize}
We define the Bessel function associated with $R$ and $k$ by, 
\begin{eqnarray*}
J_k(x,y)=\frac{1}{|W|}\sum_{w\in W} E_k(x,wy).
\end{eqnarray*}
The limit transition between hypergeometric functions $F_k$ and  Dunkl Bessel function is expressed by ( see  (2.21) of  \cite{R3} )
\begin{eqnarray}\label{FJ}
J_k(x,y)=\lim_{n\rightarrow+\infty}F_k(nx+\rho_k,\frac{y}{n})\;.
\end{eqnarray}
\par According to these preliminaries we can now formulate the  main result of this note.
\section{Integral formula for $J_k$}
The starting point is the following remarkable  integral identity obtained by \cite{AG} which connecting jack polynomials of $N$ variables to Jack polynomials of $N-1$ variables. For $\lambda=(\lambda_1,...,\lambda_N)\in \mathbb{R}^N$ we use the  notation $|\lambda|=\lambda_1+...+\lambda_N$.
\begin{prop}[\cite{AG}]
Suppose that  the partition $\mu$ has less than $N$ parts and $\lambda\in \mathbb{R}^N$ such that $\lambda_1\geq ...\geq \lambda_N$. Then
\begin{eqnarray}\label{inr}
j_\mu(\lambda)=\frac{1}{U(\mu)V(\lambda)^{2k-1}}\int_{\lambda_2}^{\lambda_1}...\int_{\lambda_N}^{\lambda_{N-1}}j_\mu(\nu)
V(\nu)\Pi(\lambda,\nu)d\nu
\end{eqnarray}
where
$$U(\mu)= \prod_{j=1}^{N-1}\beta(\mu_j+(N-j)k,k),\quad  V(\lambda)=\prod_{1\leq i<j\leq N}(\lambda_i-\lambda_j)$$
and$$\Pi(\lambda,\nu)=\prod_{ i\leq j}(\lambda_i-\nu_j)^{k-1}\prod_{ i>j}(\nu_j-\lambda_i)^{k-1}.$$
 \end{prop}
 We follow three simple steps that lead to our formula. All functions of $N$ variables will be indexed by $N$ and by  $N-1$
if it considered as $N-1$ variables.
 \\
\underline{Step 1:}
For any partition $\mu=(\mu_1,...,\mu_N)$ we set $$\widetilde{\mu}=(\mu_1-\mu_N,...,\mu_{N-1}-\mu_N,0)\quad\text{and}\quad
\overline{\mu}=(\mu_1-\mu_N,...,\mu_{N-1}-\mu_N)\in \mathbb{R}^{N-1}.$$
By Homogeneity of Jack polynomials we have that
$$j_{\mu,N}(\lambda)=\left(\prod_{j=1}^N\lambda_j\right)^{\mu_N}j_{\widetilde{\mu},N}(\lambda)$$
and from  (\ref{inr}) and (\ref{cr}) we may write
$$j_{\mu,N}(\lambda)=\frac{\left(\prod_{j=1}^N\lambda_j\right)^{\mu_N}}{U_N(\widetilde{\mu})V_N(\lambda)^{2k-1}}
\int_{\lambda_2}^{\lambda_1}...\int_{\lambda_N}^{\lambda_{N-1}}
j_{\overline{\mu},N-1}(\nu)V(\nu)\Pi(\lambda,\nu)
d\nu.$$
Taking $\lambda$ in $\mathbb{V}$ and making use  a change of variables we get that
\begin{eqnarray*}
j_{\mu,N}(e^{\lambda})&=&\frac{1}{U_N(\widetilde{\mu})V_N(e^\lambda)^{2k-1}}\int_{\lambda_2}^{\lambda_1}...\int_{\lambda_N}^{\lambda_{N-1}}
e^{|\nu|}j_{\overline{\mu},N-1}(e^{\nu})
V_{N-1}(e^{\nu})\Pi_N(e^{\lambda},e^\nu)d\nu
\\&=&\frac{1 }{U_N(\widetilde{\mu})V_N(e^\lambda)^{2k-1}}\int_{\lambda_2}^{\lambda_1}...
\int_{\lambda_N}^{\lambda_{N-1}}e^{|\nu|(1+\frac{|\overline{\mu}|}{N-1})}j_{\overline{\mu},N-1}(e^{\pi_{N-1}(\nu)})
V_{N-1}(e^{\nu})\Pi_N(e^{\lambda},e^\nu)
d\nu.
\end{eqnarray*}In order by (\ref{c}), (\ref{je}) and (\ref{har}) we have
$$F_N(\pi_N(\mu)+\rho_{k,N},\lambda)=c_N(\pi_N(\mu)+\rho_{k,N})j_{\mu,N}(e^{\lambda})=c_N(\mu+\rho_{k,N})j_{\mu,N}(e^{\lambda}),$$
where here
 $$\rho_{k,N}=\frac{k}{2}\sum_{i=1}^N(N-2i+1)e_i=\left(\frac{k(N-1)}{2},...,\frac{k(N-2i+1)}{2},...,\frac{-k(N-1)}{2}\right)\in \mathbb{R}^N$$
and
$$c_N(\mu+\rho_{k,N})=\prod_{1\leq i<j\leq N}\frac{\Gamma(\mu_i-\mu_j)\Gamma(k(j-i+1))}{\Gamma(\mu_i-\mu_j+k)\Gamma(k(j-i))}.$$
Therefore,
\begin{eqnarray*}
&&F_N(\pi_N(\mu)+\rho_{k,N},\lambda)=\frac{c_N(\mu+\rho_{k,N}) }{c_{N-1}(\overline{\mu}+\rho_{k,N-1})U_N(\widetilde{\mu})V_N(e^\lambda)^{2k-1}}\qquad\qquad\qquad\qquad\qquad\qquad\qquad\\&&\qquad\qquad\int_{\lambda_2}^{\lambda_1}...
\int_{\lambda_N}^{\lambda_{N-1}}e^{|\nu|(1+\frac{|\overline{\mu}|}{N-1})}F_{N-1}(\pi_{N-1}(\overline{\mu})+\rho_{k,N-1},\pi_{N-1}(\nu)) V_{N-1}(e^{\nu})\Pi_N(e^{\lambda},e^\nu)
d\nu.
\end{eqnarray*}
\noindent \underline{Step 2}:
Now we apply (\ref{FJ}), by using  the following when $n\rightarrow+\infty$
\begin{eqnarray*}\noindent
&&U_N(n\widetilde{\mu})\sim n^{-k(N-1) }\Gamma(k)^{N-1}\prod _{j=1}^{N-1}(\mu_j-\mu_N)^{-k},\qquad\qquad\qquad\qquad\qquad\qquad\qquad\qquad\qquad
\\&&V_N(e^{\frac{\lambda}{n}})\sim n^{-\frac{N(N-1)}{2}}V_N(\lambda),
\\&&c_N(n\mu+\rho_{k,N})\sim n^{\frac{-kN(N-1)}{2}}V_N(\mu)^{-k}
\prod_{1\leq i<j\leq N}\frac{\Gamma(k(j-i+1))}{\Gamma(k(j-i))},
\\&&c_{N-1}(n\overline{\mu}+\rho_{k,N-1})\sim n^{\frac{-k(N-1)(N-2)}{2}}V_{N-1}(\overline{\mu})^{-k}\prod_{1\leq i<j\leq N-1}\frac{\Gamma(k(j-i+1))}{\Gamma(k(j-i))},
\\&&\frac{c_N(n\mu+\rho_{k,N})}{c_{N-1}(n\overline{\mu}+\rho_{k,N-1})}\sim n^{-k(N-1)}
\frac{\Gamma(Nk)}{\Gamma (k)}\prod_{j=1}^{N-1}(\mu_j-\mu_N)^{-k},
\\&& \Pi_N(e^{\frac{\lambda}{n}}, e^{\frac{\nu}{n}})\sim n^{-N(N-1)(k-1)}\Pi(\lambda,\nu).
\end{eqnarray*}
Thus
\begin{eqnarray}\label{bd}\nonumber
J_{k,N}(\pi_N(\mu),\lambda)&=&\frac{\Gamma(Nk)}{V_N(\lambda)^{2k-1}\Gamma(k)^{N}}\int_{\lambda_2}^{\lambda_1}...
\int_{\lambda_N}^{\lambda_{N-1}}\\&&\qquad\qquad e^{|\overline{\mu}|\frac{|\nu|}{N-1}} J_{k,N-1}(\pi_{N-1}(\overline{\mu}),\pi_{N-1}(\nu))
V(\nu)\Pi(\lambda,\nu)
d\nu.
\end{eqnarray}
\underline{Step 3:}\\
The formula (\ref{bd}) is valid only for a partition $\mu$, to keep it for any $\mu\in\mathbb{R}^N$ we proceed as follows.
Let $r\in (0,+\infty)$ and $\mu$ be a partition. We obtain after a change of variables
 \begin{eqnarray*}
J_{k,N}(\pi_N(r\mu),\lambda)=J_{k,N}(\pi_N(\mu),r\lambda)&=&
  \frac{\Gamma(Nk)}{V_N(\lambda)^{2k-1}\Gamma(k)^{N}}\int_{\lambda_2}^{\lambda_1}...
\int_{\lambda_N}^{\lambda_{N-1}}e^{|\overline{r\mu}|\frac{|\nu|}{N-1}}\\&&\qquad\qquad J_{k,N-1}(\pi_{N-1}(\overline{r\mu}),\pi_{N-1}(\nu))
V(\nu)\Pi(\lambda,\nu)
d\nu.
\end{eqnarray*}
Since the set $\{r\mu;\quad r\in (0,+\infty),\quad\mu \;\text{partitions}\; \}$ is dense in the set
 $$H=\{\mu\in \mathbb{R}^N, \quad0\leq\mu_N\leq ...\leq \mu_1 \}$$ and
 $J_{k,N}$ is $S_N$-invariant continuous function then (\ref{bd}) can be extended to all $\mu\in H$. Now for $\mu\in\mathbb{R}^N$
 we denote by  $\mu^+$ the unique element of $S_N.\mu$ so that $\mu^+_N\leq...\leq\mu^+_1$. So we have
$J_{k,N}(\pi_N(\mu),\lambda)=J_{k,N}(\pi_N(\mu^+),\lambda)=J_{k,N}(\pi_N(\widetilde{\mu^+}),\lambda)$ and since $\widetilde{\mu^+}\in H$ then
\begin{eqnarray*}\label{bes}\nonumber
J_{k,N}(\pi_N(\mu),\lambda)&=&\frac{\Gamma(Nk)}{V_N(\lambda)^{2k-1}\Gamma(k)^{N}}\int_{\lambda_2}^{\lambda_1}...
\int_{\lambda_N}^{\lambda_{N-1}}\\&&\qquad\qquad e^{|\overline{\mu^+}|\frac{|\nu|}{N-1}} J_{k,N-1}(\pi_{N-1}(\overline{\mu^+}),\pi_{N-1}(\nu))
V(\nu)\Pi(\lambda,\nu)
d\nu.
\end{eqnarray*}
Now when restricted to the  space $\mathbb{V}$ we state  the following.
\begin{thm}
For all $\mu,\lambda\in\mathbb{V}$ we have
\begin{eqnarray}\label{bb}\nonumber
J_{k,N}(\mu,\lambda)&=&
  \frac{\Gamma(Nk)}{V_N(\lambda)^{2k-1}\Gamma(k)^{N}}\int_{\lambda_2^+}^{\lambda_1^+}...
\int_{\lambda_N^+}^{\lambda_{N-1}^+}e^{|\overline{\mu^+}|\frac{|\nu|}{N-1}}\\&&\qquad\qquad J_{k,N-1}(\pi_{N-1}(\overline{\mu^+}),\pi_{N-1}(\nu))
V_{N-1}(\nu)\Pi_N(\lambda^+,\nu)
d\nu.
\end{eqnarray}
\end{thm}
 In what follows, we shall restrict ourselves to the case $N=2,3,4$ where we give representations of
$J_{k,N}$ as Laplace-type integrals,
\begin{eqnarray*}
J_{k,N}(\mu,\lambda)=\int_{\mathbb{R}^N}e^{\langle\mu,x\rangle}d\nu_\lambda(x).
 \end{eqnarray*}
where $\nu_\lambda$ is a probability measure supported in the convex hall of the orbit $S_N.\lambda$.
\subsection{Bessel function of  type $A_1$}
When $N=2$ we  have that  $\mathbb{V}=\mathbb{R}(e_1-e_2)$, $\mu^+=(|\mu_1|,-|\mu_1|)$,  $\overline{\mu^+}=2|\mu_1|$ and  $\lambda^+=(|\lambda_1|,-|\lambda_1|)$. It is obvious that
 $J_{k,1}=1$, so we get from (\ref{bb})
\begin{eqnarray*}\label{bes}\nonumber
J_{k,2}(\mu,\lambda )&=& \frac{\Gamma(2k)}{(\Gamma(k))^{2}(2|\lambda_1|)^{2k-1}}\int_{-|\lambda_1|}^{|\lambda_1|}
e^{2|\mu_1|\nu}(\lambda_1^2-\nu^2)^{k-1}d\nu
\\&=&\frac{\Gamma(k+\frac{1}{2})}{\sqrt{\pi}\Gamma(k)}\int_{-1}^{1}
e^{(2|\mu_1||\lambda_1|\nu}(1-\nu^2)^{k-1}d\nu
\\&=&\mathcal{J}_{k-\frac{1}{2}}(2\mu_1\lambda_1)
\end{eqnarray*}
where $\mathcal{J}_{k-\frac{1}{2}}$ is the modified  Bessel function given by
$$\mathcal{J}_{k-\frac{1}{2}}(z)=\Gamma(k+\frac{1}{2})\sum_{n=0}^\infty \frac{1}{n!\Gamma(n+k+\frac{1}{2})}(\frac{z}{2})^{2n}.$$
However, it is usual to  identify   $\mathbb{V}=\mathbb{R}\varepsilon$,  $\varepsilon=\frac{e_1-e_2}{\sqrt{2}}$ with $\mathbb{R}$ and write
\begin{eqnarray*}
J_{k,2}(\mu,\lambda )=\mathcal{J}_{k-\frac{1}{2}}(\mu\lambda),\qquad
\mu,\lambda\in\mathbb{R}.
\end{eqnarray*}
\subsection{Bessel function of type $A_2$}
Let $\mu=(\mu_1,\mu_2,\mu_3)$ and $\lambda=(\lambda_1,\lambda_2,\lambda_3)$ in the fundamental Weyl chamber
$$C=\{(u_1,u_2,u_3);\quad u_1\geq u_2\geq u_3,\quad u_1+u_2+u_3=0\}.$$
With $\overline{\mu}=(\mu_1-\mu_3,\mu_2-\mu_3, )$ and $\pi_2(\overline{\mu})=(\frac{\mu_1-\mu_2}{2}, \frac{\mu_2-\mu_1}{2})$
the formula (\ref{bb}) gives
\begin{eqnarray*}
J_{k,3}(\mu,\lambda)&=&\frac{\Gamma(3k)}{V(\lambda)^{2k-1}\Gamma(k)^3}\int_{\lambda_2}^{\lambda_1}
\int_{\lambda_3}^{\lambda_{2}}e^{\frac{(\mu_1+\mu_2-2\mu_3)(\nu_1+\nu_2)}{2}}
\mathcal{J}_{k-\frac{1}{2}}(\frac{(\mu_1-\mu_2)(\nu_1-\nu_2)}{2})(\nu_1-\nu_2)
\\&&\qquad\Big((\lambda_1-\nu_1)(\lambda_1-\nu_2)
(\lambda_2-\nu_2)(\nu_1-\lambda_2)(\nu_1-\lambda_3)(\nu_2-\lambda_3)\Big)^{k-1}d\nu_1d\nu_2.
\end{eqnarray*}
Using the
 change of variables: $x=\frac{\nu_1+\nu_2}{2}$,  $z=\frac{\nu_1-\nu_2}{2}$ we have
\begin{eqnarray}\label{1}\nonumber
J_{k,3}(\mu,\lambda)=
&& \frac{4\Gamma(3k)}{V(\lambda)^{2k-1}\Gamma(k)^3}
\int_{\mathbb{R} }\int_{\mathbb{R} }ze^{(\mu_1+\mu_2-2\mu_3)x}\mathcal{J}_{k-\frac{1}{2}}(\mu_1-\mu_2)z) \; \chi_{[\lambda_2,\lambda_1]}(x+z)\;\chi_{[\lambda_3,\lambda_2]}(x-z)\\&&
\Big((\lambda_1-x)^2-z^2)((\lambda_3-x)^2-z^2)(z^2-(\lambda_2-x)^2)\Big)^{k-1}dx dz.
\end{eqnarray}
Now recall that
\begin{eqnarray}\label{2}\nonumber
\mathcal{J}_{k-\frac{1}{2}}((\mu_1-\mu_2)z)&=&\frac{\Gamma(2k)}{2^{2k-1}\Gamma
(k)^2}\int_{\mathbb{R}}
e^{(\mu_1-\mu_2)zt}(1-t^2)^{k-1}\chi_{[-1,1]}(t)dt.
\\&=&\frac{\Gamma(2k)}{2^{2k-1}\Gamma(k)^2}\int_{\mathbb{R}}
e^{(\mu_1-\mu_2)y}(1-\frac{y^2}{z^2})^{k-1}\chi_{[-1,1]}(\frac{y}{z})z^{-1}dy
\end{eqnarray}
then inserting (\ref{2}) in (\ref{1}) with  the use of  Fubini's Theorem we can write
\begin{eqnarray*}
J_{k,3}(\mu,\lambda)= \int_{\mathbb{R} }\int_{\mathbb{R} }
 e^{(\mu_1+\mu_2-2\mu_3)x+ (\mu_1-\mu_2)y} \Delta_k(\lambda,x,y)dxdy
\end{eqnarray*}
 where
\begin{eqnarray*}
&&\Delta_k(\lambda,x,y)=\\&&
\frac{4\Gamma(2k)\Gamma(3k)}{2^{2k-3}\Gamma(k)^5V(\lambda)^{2k-1}}\int_{\mathbb{R} }\left(\frac{z^2-y^2}{z^2}\right)^{k-1}
  \Big((\lambda_1-x)^2-z^2)((\lambda_3-x)^2-z^2)(z^2-(\lambda_2-x)^2)\Big)^{k-1}\\&&\qquad\qquad\qquad\qquad\qquad\qquad\qquad\chi_{[-1,1]} (\frac{y}{z}) \chi_{[\lambda_1,\lambda_2]}(x+z)\chi_{[\lambda_3,\lambda_2]}(x-z)dz
\end{eqnarray*}
We  note here  that
\begin{eqnarray*}
\chi_{[-1,1]} (\frac{y}{z}) \chi_{[\lambda_1,\lambda_2]}(x+z)\chi_{[\lambda_3,\lambda_2]}(x-z)=
\chi_{\max(|y|,|x-\lambda_2|)\leq z\leq \min(x-\lambda_3, \lambda_1-x)}.
\end{eqnarray*}
Thus we have \begin{eqnarray*}\nonumber
\Delta_k(\lambda,x,y)=&&\frac{4\Gamma(2k)\Gamma(3k)}{2^{2k-3}\Gamma(k)^5V(\lambda)^{2k-1}}
\int_{\max(|y|,|x-\lambda_2|)}^{\min(x-\lambda_3, \lambda_1-x)} \left(\frac{z^2-y^2}{z^2}\right)^{k-1}\qquad\qquad\\&&\qquad
  \Big((\lambda_1-x)^2-z^2)((\lambda_3-x)^2-z^2)(z^2-(\lambda_2-x)^2)\Big)^{k-1}dz\label{del}
\end{eqnarray*}
if
\begin{eqnarray*}\label{supp}
\max(|y|,|x-\lambda_2|)\leq \min(x-\lambda_3, \lambda_1-x)
\end{eqnarray*}
and $\Delta_k(\lambda,x,y)= 0$, otherwise.
Making the change of variables
$$ \nu_1=x+y,\qquad  \nu_2=x-y$$
and put $\nu=(\nu_1,\nu_2,\nu_3)\in \mathbb{V}$ with $\nu_3=-(\nu_1+\nu_2)$ we obtain
\begin{eqnarray*}
J_3^k(\mu,\lambda)= \frac{1}{2}\int_{\mathbb{R}^2 }
 e^{\mu_1\nu_1+\mu_2\nu_2+\mu_3\nu_3} \Delta_{k,2}\left(\lambda,\frac{\nu_1+\nu_2}{2},\frac{\nu_1-\nu_2}{2}\right)d\nu_1d\nu_2
\end{eqnarray*}
But we can identify  $\mathbb{R}^2$ with the space $\mathbb{V}$  via the basis $(e_1-e_2, e_2-e_3)$, since for $\nu=(\nu_1,\nu_2,\nu_3)\in \mathbb{V}$ we have
$\nu= \nu_1(e_1-e_2)+ \nu_2(e_1-e_3)$. Then we get
\begin{eqnarray}\label{jf}
J_{k,3}(\mu,\lambda)= \int_{\mathbb{R}^2 }
 e^{\langle\mu,\nu\rangle} \delta_{k,2}\left(\lambda,\nu\right)d\nu_1d\nu_2.
\end{eqnarray}
with
\begin{eqnarray*}
\delta_{k,2}(\lambda,\nu)=\frac{1}{2}\Delta_{k,2}\left(\lambda,\frac{\nu_1+\nu_2}{2},\frac{\nu_1-\nu_2}{2}\right).
\end{eqnarray*}
   \par  Now  considering  the orthonormal basis $(\varepsilon_1,\varepsilon_2)$ of $\mathbb{V}$,
 $$\varepsilon_1=\frac{1}{\sqrt{6}}(e_1+e_2-2e_3),\quad \varepsilon_2=\frac{1}{\sqrt{2}}(e_1-e_2)$$
we can write
$$\mu=\frac{(\mu_1+\mu_2-2\mu_3)}{\sqrt{6}}\;\varepsilon_1+ \frac{\mu_1-\mu_2} {\sqrt{2}}\;\varepsilon_2$$
and for  $x=x_1\varepsilon_1+x_2\varepsilon_2$
$$\langle\mu,x\rangle=\frac{(\mu_1+\mu_2-2\mu_3)}{\sqrt{6}}x_1+\frac{\mu_1-\mu_2} {\sqrt{2}}x_2$$
Then
using  change of variables $x_1=\sqrt{6}\;x$ and $x_2=\sqrt{2}\;y$ in the formula (\ref{k3}) we obtain
\begin{eqnarray*}
J_3^k(\mu,\lambda)= \frac{1}{\sqrt{12}}\int_{\mathbb{R}^2 }
 e^{\langle\mu,x\rangle} \Delta_k\left(\lambda,\frac{x_1}{\sqrt{6}},\frac{x_2}{\sqrt{2}}\right)dx_1dx_2.
\end{eqnarray*}
\begin{prop} For all $\lambda=(\lambda_1,\lambda_2,\lambda_3)\in C$ the function $\delta_{k,2}\left(\lambda,.\right)$ is supported in the
closed convex hull $co(\lambda)$ of the $S_3$-orbit of $\lambda$, described by:
$\nu=(\nu_1,\nu_2,\nu_3)\in \mathbb{V}$ such that
\begin{eqnarray*}
\lambda_3\leq\min(\nu_1,\nu_2,\nu_3)\leq \max(\nu_1,\nu_2,\nu_3)\leq \lambda_1.
\end{eqnarray*}
\end{prop}
\begin{proof}
In view of (\ref{supp}) and (\ref{del}) the support of $\delta_{k,2}\left(\lambda,.\right)$  is contain is the set
$$\left\{\nu\in \mathbb{V}; \quad \max \left(\frac{|\nu_1-\nu_2|}{2},\left| \frac{\nu_1+\nu_2}{2}-\lambda_2\right|\right)\leq \min
\Big(\frac{\nu_1+\nu_2}{2}-\lambda_3,\lambda_1-\frac{\nu_1+\nu_2}{2}\Big)\right\}$$
which by straightforward calculus reduced to the set
$$\{\nu\in \mathbb{V};\quad\lambda_3\leq\min(\nu_1,\nu_2,\nu_3)\leq \max(\nu_1,\nu_2,\nu_3)\leq \lambda_1\}.
$$
However, we  known that
$$\nu\in co(\lambda)\quad\Leftrightarrow \quad \lambda^+-\nu^+ \in \bigoplus_{i=1}^N\mathbb{R}_+\alpha_i$$
and  here
$$\lambda^+-\nu^+= \lambda-\nu^+=(\lambda_1-\nu^+_1)(e_1-e_2)+(\nu^+_3-\lambda_3)(e_2-e_3)$$
Then
$$\nu\in co(\lambda)\quad\Leftrightarrow \quad \nu^+_1\leq \lambda_1 \quad \text{and} \quad \nu^+_3\geq \lambda_3,$$
which proves the proposition, since $\nu^+_1=\max(\nu_1,\nu_2,\nu_3)$ and $\nu^+_3= \min(\nu_1,\nu_2,\nu_3)$.
\end{proof}

\subsection{Bessel function of type $A_3$}
Let $\mu,\lambda \in C$, the Weyl chamber. We have
\begin{eqnarray*}
|\overline{\mu}|&=&\mu_1+\mu_2+\mu_3-3\mu_4,\\
\pi_4(\overline{\mu})&=&(\mu_1+\frac{\mu_4}{3},\mu_2+\frac{\mu_4}{3},\mu_3+\frac{\mu_4}{3}),\\
\pi_4(\nu)&=&(\frac{2\nu_1-\nu_2-\nu_3}{3},\frac{2\nu_2-\nu_1-\nu_3}{3},\frac{2\nu_3-\nu_1-\nu_2}{3}).
\end{eqnarray*}
Taking (\ref{bb}) with  the change of variables
\begin{eqnarray*}
z_1&=&\frac{\nu_1+\nu_2+\nu_3}{3},\\
x_1&=&\frac{2\nu_1-\nu_2-\nu_3}{3},
\\x_2&=&\frac{2\nu_2-\nu_1-\nu_3}{3}
\end{eqnarray*}
and put $x=(x_1,x_2, x_3)\in \mathbb{R}^3$,  $x_1+x_2+x_3=0$, we have
\begin{eqnarray*}
&&J_{k,4}(\mu,\lambda)=\int_{\mathbb{R}^3}e^{(\mu_1+\mu_2+\mu_3-3\mu_4)z_1}
J_{k,3}(\pi_3(\overline{\mu}), x)  V_3(x)\\&& \Pi_4(x_1+z_1,x_2+z_1,x_3+z_1,\lambda)\chi_{[\lambda_2,\lambda_1]}(x_1+z_1)
\chi_{[\lambda_3,\lambda_2]}(x_2+z_1)\chi_{[\lambda_4,\lambda_3]}(x_3+z_1)dz_1dx_1dx_2.
\end{eqnarray*}
By inserting (\ref{jf})
\begin{eqnarray*}
J_{k,4}(\mu,\lambda)&=&\int_{\mathbb{R}^5}e^{\mu_1+\mu_2+\mu_3-3\mu_4)z_1+(\mu_1-\mu_3)z_2+(\mu_2-\mu_3)z_3}
\delta_{k,2} ((z_2,z_3,-(z_2+z_3)),x ) V_3(x)\\&& \Pi_4(x_1+z_1,x_2+z_1,x_3+z_1,\lambda)\chi_{[\lambda_2,\lambda_2]}(x_1+z_1)
\chi_{[\lambda_3,\lambda_2]}(x_2+z_1)\chi_{[\lambda_4,\lambda_3]}(x_3+z_1)\\&&\qquad\qquad\qquad
\qquad\qquad\qquad \qquad\qquad\qquad\qquad\qquad\qquad\qquad dz_1dz_2dz_3dx_1dx_2.
\end{eqnarray*}
Now with the change of variables
\begin{eqnarray*}
Z_1&=&z_1+z_2,\\
Z_2&=&z_1+z_3,\\
Z_3&=&z_1-(z_2+z_3)
\end{eqnarray*}
and with $Z=(Z_1,Z_2,Z_3,Z_4)\in \mathbb{R}^4$, such that $Z_1+Z_2+Z_3+Z_4=0$ we have that
$$(\mu_1+\mu_2+\mu_3-3\mu_4)z_1+(\mu_1-\mu_3)z_2+(\mu_2-\mu_3)z_3= \mu_1Z_1+\mu_2Z_2+\mu_3Z_3+\mu_4Z_4=\langle\mu,Z\rangle.$$
Therefore we can write
\begin{eqnarray*}
J_{k,3}(\mu,\lambda)=\int_{\mathbb{R}^3}e^{\langle \mu,Z\rangle}\delta_{k,3}(Z,\lambda) dZ_1dZ_2dZ_3,
\end{eqnarray*}
where
\begin{eqnarray}\label{d2} \nonumber
&&\delta_{k,3}(Z,\lambda)=\\ \nonumber &&\int_{\mathbb{R}^2}\Pi_4(x_1+\frac{1}{3}(Z_1+Z_2+Z_3),x_2+\frac{1}{3}(Z_1+Z_2+Z_3),x_3+\frac{1}{3}(Z_1+Z_2+Z_3),\lambda)\\&&
\delta_{k,2}(\frac{1}{3}(2Z_1-Z_2-Z_3),\frac{1}{3}(2Z_2-Z_1-Z_3),\frac{1}{3}(2Z_3-Z_1-Z_2),x)\nonumber
\chi_{[\lambda_2,\lambda_1]}(x_1+\frac{1}{3}(Z_1+Z_2+Z_3))\\&&
\chi_{[\lambda_3,\lambda_2]}(x_2+\frac{1}{3}(Z_1+Z_2+Z_3))\chi_{[\lambda_4,\lambda_3]}(x_3+\frac{1}{3}(Z_1+Z_2+Z_3))dx_1dx_2.
\end{eqnarray}
Let us now describe the support of $\delta_{k,3}$. In fact, $\delta_{k,3}(Z,\lambda)\neq 0$ if the variables $x$ and $Z$ of the integrant (\ref{d2})
satisfy:
\begin{eqnarray*}
&(1)&\quad\lambda_2\leq x_1+\frac{1}{3}(Z_1+Z_2+Z_3)\leq \lambda_1,\\
&(2)&\quad\lambda_3\leq x_2+\frac{1}{3}(Z_1+Z_2+Z_3)\leq \lambda_2,\\
&(3)&\quad\lambda_4\leq x_3+\frac{1}{3}(Z_1+Z_2+Z_3)\leq \lambda_3,\\
&(4)&\quad x_3\leq \frac{1}{3}(2Z_1-Z_2-Z_3)\leq x_1,\\
&(5)&\quad x_3\leq \frac{1}{3}(2Z_2-Z_1-Z_3)\leq x_1,\\
&(6)&\quad x_3\leq \frac{1}{3}(2Z_3-Z_1-Z_2)\leq x_1.
\end{eqnarray*}
It Follows that
\begin{eqnarray*}
 (1)+(4)\qquad&\Rightarrow&\qquad Z_1\leq \lambda_1,\\
 (1)+(5)\qquad&\Rightarrow&\qquad Z_2\leq \lambda_1,\\
 (1)+(6)\qquad&\Rightarrow&\qquad Z_3\leq \lambda_1,\\
 (1)+(2)+(3)\qquad&\Rightarrow&\qquad Z_4\leq \lambda_1,\\
 (1)+(2)-(6)\qquad&\Rightarrow&\qquad Z_1+Z_2\leq \lambda_1+\lambda_2,\\
(1)+(2)-(5)\qquad&\Rightarrow&\qquad Z_1+Z_3\leq \lambda_1+\lambda_2,\\
(1)+(2)-(4)\qquad&\Rightarrow&\qquad Z_2+Z_3\leq \lambda_1+\lambda_2,\\
(2)+(3)-(4)\qquad&\Rightarrow&\qquad Z_1+Z_4\leq \lambda_1+\lambda_2,\\
(2)+(3)-(5)\qquad&\Rightarrow&\qquad Z_2+Z_4\leq \lambda_1+\lambda_2\\
(2)+(3)-(6)\qquad&\Rightarrow&\qquad Z_3+Z_4\leq \lambda_1+\lambda_2,\\
(1)+(2)+(3)\qquad&\Rightarrow&\qquad Z_1+Z_2+Z_3\leq \lambda_1+\lambda_2+\lambda_3,\\
(3)+(4)\qquad&\Rightarrow&\qquad Z_2+Z_3+Z_4\leq \lambda_1+\lambda_2+\lambda_3,\\
(3)+(5)\qquad&\Rightarrow&\qquad Z_1+Z_3+Z_4\leq \lambda_1+\lambda_2+\lambda_3,\\
(3)+(6)\qquad&\Rightarrow&\qquad Z_1+Z_2+Z_4\leq \lambda_1+\lambda_2+\lambda_3.
\end{eqnarray*}
These inequalities   can be expressed in terms  of $Z^+=(Z^+_1,Z^+_2,Z^+_3,Z^+_4)$ as
\begin{eqnarray*}
&&Z^+_1=\max(Z_1,Z_2,Z_3,Z_4)\leq\lambda_1\\
&&Z^+_1+Z_2^+=\max(Z_1+Z_2,Z_1+Z_3,Z_1+Z_4,Z_2+Z_3,Z_2+Z_4,Z_3+Z_4)\leq\lambda_1+\lambda_2\\
&&Z^+_1+Z_2^++Z^+_3=\max(Z_1+Z_2+Z_3,Z_2+Z_3+Z_4,Z_1+Z_2+Z_4,Z_1+Z_3+Z_4)\\&&\qquad\qquad\qquad\qquad\qquad\qquad
\qquad\qquad\qquad\qquad\qquad\qquad\qquad\qquad\qquad\leq\lambda_1+\lambda_2+\lambda_3
\end{eqnarray*}
which imply that   $Z^+\preceq \lambda$ and therefore $Z\in co(\lambda)$.
\subsection{Case for arbitrary $N$}
After having idea about the case $N=2,3$ it is not hard to see that the formula (\ref{0}) can  be found using recurrence. In fact,
let $\mu,\lambda\in C$ the Weyl chamber. Put for $\nu\in\mathbb{R}^{N-1}$
$$\Omega(\lambda,\nu)=\prod_{i=1}^{N-1}\chi_{[\lambda_{i+1},\lambda_i]}(\nu).$$
With the  change of variables
\begin{eqnarray*}
&&z_1=\frac{|\nu|}{N-1}=\frac{\nu_1+...+\nu_{N-1}}{N-1}
\\ &&x_i=\nu_i-\frac{|\nu|}{N-1};\qquad 1\leq i\leq N-2
\end{eqnarray*}
the formula  (\ref{bb}) becomes,
\begin{eqnarray*}\nonumber
J_{k,N}(\mu,\lambda)
&=&\frac{\Gamma(Nk)}{V_N(\lambda)^{2k-1}\Gamma(k)^{N}}\int_{\mathbb{R}^{N-1}}e^{|\overline{\mu}|z_1}J_{k,N-1}(\pi_{N-1}(\overline{\mu}),x)
V_{N-1}(x)\\&&\Pi_N(\lambda, (x_1+z_1,...,x_{N-1}+z_1))\Omega(\lambda,(x_1+z_1,...,x_{N-1}+z_1)) dz_1dx_1...dx_{N-2}
\end{eqnarray*}
where we put $x=(x_1,...,x_{N-1})$ with $x_{N-1}=-(x_1+...+x_{N-2})$. The recurrence hypothesis says that
\begin{eqnarray*}
J_{k,N-1}(\pi_{N-1}(\overline{\mu}),x)&=&\int_{\mathbb{V}_{N-1}}e^{\langle\pi_{N-1}(\overline{\mu}),z\rangle}\delta_{k,N-1}(x,z)dz.\\
&=& \int_{\mathbb{R}^{N-2}}e^{\sum_{i=1}^{N-1}\left(\overline{\mu}_i-\frac{|\overline{\mu}|}{N-1}\right)z_{i+1}}\delta_{k,N-1}(x,z)dz_2...dz_{N-2}.
\end{eqnarray*}
where  $z=(z_2,...,z_{N})$ with $z_{N}=-(z_2+...+z_{N-1})$ and $\delta_{k,N-1}(.,x)$ is supported in the convex hull of
$S_{N-1}.x$ in $\mathbb{R}^{N-1}$.
Hence  we get
\begin{eqnarray*}\nonumber
J_{k,N}(\mu,\lambda)
&=&\frac{\Gamma(Nk)}{V_N(\lambda)^{2k-1}\Gamma(k)^{N}}\int_{\mathbb{R}^{N-1}}\int_{\mathbb{R}^{N-2}}
e^{|\overline{\mu}|z_1+\sum_{i=1}^{N-1}\left(\overline{\mu}_i-\frac{|\overline{\mu}|}{N-1}\right)z_{i+1}} \delta_{k,N-1}(z,x)\\&&
V_{N-1}(x)\Pi_N(\lambda, (x_1+z_1,...,x_{N-1}+z_1)) \Omega(\lambda,(x_1+z_1,...,x_{N-1}+z_1))\\&&\qquad\qquad\qquad\qquad dz_1dz_2...dz_{N-1}dx_1...dx_{N-2}.
\end{eqnarray*}
Now observing that
\begin{eqnarray*}
|\overline{\mu}|z_1+\sum_{i=1}^{^{N-1}}\left(\overline{\mu}_i-\frac{|\overline{\mu}|}{N-1}\right)z_{i+1}&=&
 \left(\sum_{i=1}^{N-1}\mu_i-(N-1)\mu_N\right)z_1+\sum_{i=1}^{N-1}\mu_i z_{i+1}\\
&=&\sum_{i=1}^{N-1}\mu_i(z_1+z_{i+1})-(N-1)\mu_{N-1}z_1.
\end{eqnarray*}
Then making the change of variables
\begin{eqnarray*}
Z_i=z_1+z_{i+1},\qquad 1\leq i\leq N-1
\end{eqnarray*}
and put $Z=(Z_1,...,Z_N)$ with $Z_N=-(Z_1+...+Z_{N-1}$, we have
\begin{eqnarray*}
J_{k,N}(\mu,\lambda)
&=&\frac{\Gamma(Nk)}{V_N(\lambda)^{2k-1}\Gamma(k)^{N}}\int_{\mathbb{R}^{N-1}}\int_{\mathbb{R}^{N-2}}
e^{\sum_{i=1}^N\mu_iZ_i} \delta_{k,N-1}(\phi(Z),x)
V_{N-1}(x)\Pi_N(\lambda, \theta(Z,x))\\&&\qquad\qquad\qquad\Omega_N(\lambda, \theta(Z,x))dZ_1dZ_2...dZ_{N-1}dx_1...dx_{N-2}\\
&=&\int_{\mathbb{R}^{N-1}}e^{\sum_{i=1}^N\mu_iZ_i} \delta_{k,N}(\lambda,Z)dZ_1...dZ_{N-1}
\\&=&\int_{\mathbb{V}_{N}}e^{\langle\mu,Z\rangle} \delta_{k,N}(\lambda,Z)dZ,
\end{eqnarray*}
with
\begin{eqnarray}\nonumber
\phi(Z)&=&\left(Z_1-\frac{\sum_{i=1}^{N-1} Z_i}{N-1},...,Z_{N_{-1}}-\frac{\sum_{i=1}^{N-1} Z_i}{N-1}\right)\\
\theta(Z,x)&=&\left(x_1+\frac{\sum_{i=1}^{N-1} Z_i}{N-1},...,x_{N-1}+\frac{\sum_{i=1}^{N-1} Z_i}{N-1}\right)\nonumber \\ \nonumber
\delta_{k,N}(\lambda,Z)&=&\frac{\Gamma(Nk)}{V_N(\lambda)^{2k-1}\Gamma(k)^{N}}\int_{\mathbb{R}^{N-2}}
 \delta_{k,N-1}(\phi(Z),x)\\&&\qquad\qquad\qquad
V_{N-1}(x)\Pi_N(\lambda, \theta(Z,x))\Omega_N(\lambda, \theta(Z,x))dx_1...dx_{N-2} \label{dn}.
\end{eqnarray}
Now we write sufficient conditions for which  the integrant (\ref{dn}) does not vanish
\begin{eqnarray*}
&(\Lambda_i)&\qquad \lambda_{i+1}\leq x_i+\frac{\sum_{i=1}^{N-1} Z_i}{N-1}\leq \lambda_i\\
&(\Lambda_I)&\qquad  \sum_{i\in I} Z_i-\frac{|I|}{N-1}\sum_{i=1}^{N-1} Z_i\leq \sum_{i=1}^{|I|}x_i
\end{eqnarray*}
for all $I\subset\{1,2,...,N-1\}$ of cardinally $|I|$. It follows that
$$\sum_{i=1}^{|I|}\Lambda_i+\Lambda_I\quad \Rightarrow\quad \sum_{i\in I}Z_i\leq \sum_{i=1}^{|I|}\lambda_i$$
which proves that $Z^+\leq \lambda$ and then $Z\in co(\lambda)$.
\section{Partially product formula for $J_k$ }
We will first  establish a product formula for $J_k$ provided that a conjecture of Stanley on the multiplication of Jack polynomials is true. The conjecture says that for all partitions $\mu$ and $\lambda$
\begin{eqnarray*}
j_\mu j_\lambda=\sum _{\nu\leq \mu+\lambda}g_{\mu,\lambda}^\nu j_\nu
\end{eqnarray*}
where  $g_{\mu,\lambda}^\nu$ ( the Littlewood-Richardson coefficients )
is a polynomial in $k$ with nonnegative integer coefficients. In particular, $g_{\mu,\lambda}^\nu\geq 0$, what is the interesting facts in our setting.
Hence  we have for all $\mu,\lambda$ partitions,
 \begin{eqnarray*}
 F(\pi(\mu)+\rho_k,.)F(\pi(\lambda)+\rho_k,.)= \sum _{\nu\leq \mu+\lambda} f_{\mu,\lambda}^\nu F(\pi(\nu)+\rho_k,.)
 \end{eqnarray*}
with $f_{\mu,\lambda}^\nu\geq 0$  and
$$\sum_\nu f_{\mu,\lambda}^\nu =1$$
But if $\nu\leq \mu+\lambda$ as partitions then we also have  $\pi(\nu)\preceq \pi(\mu)+\pi(\lambda)$ in the dominance ordering (\cite{Br}, Lemma 3.1 ).
This allows us to write for all $\mu,\lambda\in P^+$
\begin{eqnarray*}
 F(\mu+\rho_k,.)F(\lambda+\rho_k,.)= \sum _{\nu\in P^+;\; \nu\preceq \mu+\lambda} f_{\mu,\lambda}^\nu F(\nu+\rho_k,.).
 \end{eqnarray*}
 To arrive at product formula for $J_k$ we follow the technic used  by M. R\"{o}sler in \cite{R3}. We first write
$$F(n\mu+\rho_k,\frac{z}{n})F(n\lambda+\rho_k,\frac{z}{n})=\int_{\mathbb{R}^N} F(nx+\rho_k,\frac{z}{n})d\gamma_{\mu,\lambda}^n(x), \qquad z\in \mathbb{V}.$$
where
$$d\gamma_{\mu,\lambda}^n=\sum_{\nu\in P^+;\; \nu\preceq \mu+\lambda}f_{\mu,\lambda}^\nu\; \delta_{\frac{\nu}{n}}.$$
According to (\cite{R3}, Lemma 3.2) the  probability measure $ \gamma_{\mu,\lambda}^n$ is  supported in the convex hull $co(\mu+\lambda)$. So, from
  Prohorov's theorem (see \cite{Bi} ) there exists a  probability measure $\gamma_{\mu,\lambda}$ supported in $co(\mu+\lambda)$ and  a subsequence $(\gamma_{\mu,\lambda}^{n_j})_j $ which converges weakly to $\gamma_{\mu,\lambda}$.
  Then by using (\ref{FJ}) it follows that
 $$J_k(\mu,z)J_k(\lambda,z)=\int_{\mathbb{V}}J_k(\xi,z)d\gamma_{\mu,\lambda}(\xi)$$
for all $ z\in\mathbb{V}$ and $\mu,\lambda\in P^+$.\\ Now let $r,s\in \mathbb{Q}^+$ with $r=\frac{a}{b}$ and $s=\frac{c}{b}$,  $a,b,c \in\mathbb{N}$, $b\neq0$. We write
\begin{eqnarray*}
J_k(r\mu,z)J_k(s\lambda,z)&=&J_k(a\mu,\frac{z}{b})J_k(c\lambda,\frac{z}{b})\\
&=&\int_{\mathbb{V}}J_k(\xi,\frac{z}{b})d\gamma_{a\mu,c\lambda}(\xi);\quad z\in\mathbb{R}^d \\
&=&\int_{\mathbb{V}}J_k(\frac{\xi}{b},z)d\gamma_{a\mu,c\lambda}(\xi);\quad z\in\mathbb{R}^d
\end{eqnarray*}
Defining $\gamma_{r\mu,s\lambda}$
as the image measure of $\gamma_{a\mu,c\lambda}$ under the dilation $\xi \rightarrow\frac{\xi}{b}$. We get
$$J_k(r\mu,z)J_k(s\lambda,z)=
 \int_{\mathbb{V}}J_k(\xi,z)d\gamma_{r\mu, s\lambda}(\xi). $$
Now we apply the density argument, since  $\mathbb{Q}^+.P^+\times \mathbb{Q}^+.P^+$ is dense in $C\times C$, where $C$ is the
 Weyl chamber.
Then  Prohorov's theorem yields
 \begin{eqnarray*}
J_k(\mu,z)J_k(\lambda,z)=\int_{\mathbb{V}}J_k(\xi,z)d\gamma_{\mu,\lambda}(\xi);\quad z\in\mathbb{R}^d
\end{eqnarray*}
for all $\mu,\lambda\in C$ with $supp(\gamma_{\mu,\lambda})\subset co(\mu+\lambda)$. This finish our approach for the product formula.
\par
  An important special case of the Stanley conjecture called Peiri formula is where the partition $\lambda=(n)$, $n\in \mathbb{N}$.
Since this formula has already been proved  (see \cite{S}) then we can state the following partial result
\begin{thm}
For all $\mu\in C$ and all $t\geq 0$ there exists a probability measure $\gamma_{\mu,t}$ such that
\begin{eqnarray*}
J_k(\mu,z)J_k(t\beta_1,z)=\int_{\mathbb{V}}J_k(\xi,z)d\gamma_{\mu,t}(\xi);\quad z\in\mathbb{R}^d
\end{eqnarray*}
where $\beta_1=\pi(e_1)$. The measure $\gamma_{\mu,t}$ is supported in $co(\mu+t\beta_1)$.
\end{thm}

\end{document}